\documentclass{amsart}[12 pt]

\usepackage{latexsym}
\usepackage{amsmath,amssymb,dsfont,amsfonts}
\usepackage{mathrsfs}
\usepackage{verbatim}
\usepackage{graphicx}
\usepackage{graphics}
\usepackage{geometry}
\usepackage{booktabs}
\usepackage{enumitem}
\usepackage{xcolor}

\newtheorem{theorem}{Theorem}[section]

\newtheorem{lemma}{Lemma}[section]

\numberwithin{equation}{section}

\allowdisplaybreaks

\begin{document}

\markboth{}{}

\title[$L^p$ uniform random walk-type approximation for Fractional Brownian motion] {$L^p$ uniform random walk-type approximation for Fractional Brownian motion with Hurst exponent $0 < H < \frac{1}{2}$}

\author{Alberto Ohashi}

\address{Departamento de Matem\'atica, Universidade de Bras\'ilia, 70910-900, Bras\'ilia, Brazil.}\email{amfohashi@gmail.com}

\author{Francys A. de Souza}

%%% ADAPT THE ADDRESS OF LEAO
\address{Mathematics Department, Rua S´ergio Buarque de Holanda, 651, UNICAMP - State University of Campinas, 13083-859 Campinas, Brazil.
}\email{francys@estatcamp.com.br}

\date{\today}


\begin{center}
\end{center}
\begin{abstract}
In this note, we prove an $L^p$ uniform approximation of the fractional Brownian motion with Hurst exponent $0 < H < \frac{1}{2}$ by means of a family of continuous-time random walks imbedded on a given Brownian motion. The approximation is constructed via a pathwise representation of the fractional Brownian motion in terms of a standard Brownian motion. For an arbitrary choice $\epsilon_k$ for the size of the jumps of the family of random walks, the rate of convergence of the approximation scheme is $O(\epsilon_k^{p(1-2\lambda)+ 2(\delta-1)})$ whenever $\max\{0,1-\frac{pH}{2}\}< \delta < 1$, $\lambda \in \big(\frac{1-H}{2}, \frac{1}{2} + \frac{\delta-1}{p}\big)$.
\end{abstract}
\maketitle

\section{Introduction}
The fractional Brownian motion $B_H$ (henceforth abbreviated by FBM) with Hurst exponent $H \in (0,1)$ is the zero mean Gaussian process with covariance function $\mathbb{E}[B_H(t)B_H(s)] = \frac{1}{2}[s^{2H} + t^{2H}  - |t-s|^{2H}]$. It turns out that FBM is the only continuous Gaussian process which is self-similar with stationary increments. There are many applications of FBM in sciences, including Physics,
Biology, Hydrology, network research, Finance; see Biagini et al. \cite{biagini} and other references therein. In Probability theory, FBM is the canonical model of Gaussian process which exhibits non-trivial increment correlations (for $H \neq \frac{1}{2}$) and it is still amenable to rigorous modelling by means of Malliavin calculus and rough path techniques. The goal of this note is to present $L^p$ uniform approximations for the FBM with Hurst exponent $0 < H < \frac{1}{2}$ by means of a family of continuous time random walks.

The study of approximations of FBM (in the sense of weak convergence) dates back from 1970s with the pioneering works of Davydov \cite{davydov} and Taqqu \cite{taqqu}. Since then, many authors have been proposed alternative weak approximation methods based on correlated random walks, random wavelet series, Poisson processes, etc. In this direction, we refer the reader to Bardina et al \cite{bardina}, Delgado and Jolis \cite{delgado}, Enriquez \cite{enriquez},  Kl\"uppelberg and K\"uhn \cite{klu} and Li and Dai \cite{dai} and other references therein. Almost sure uniform approximations for FBM have been studied by many authors in different contexts via transport processes, series representations, etc. In this direction, we refer the reader to Garzon et al \cite{garzon}, Hong et al \cite{hong},  Dzhaparidze and Van Zanten \cite{dzh}, Chen and Dong \cite{chen}, Igloi \cite{igloi} and other references therein. Other approximations in $L^p(\Omega\times [0,T])$ were proposed by Decreusefond and Ustunel \cite{ustunel} and $L^p$-estimates (not uniform in time) by Mishura \cite{mishura}.

In this work, we present an $L^p$-approximation for a given FBM with Hurst exponent $H \in (0,\frac{1}{2})$ in the supremum norm over a given time interval $[0,T]$. Motivated by the stochastic analysis of non-Markovian phenomena in the rough regime, we restrict our analysis to the most delicate case $0 < H < \frac{1}{2}$. An important step in our analysis is a pathwise representation $B_H = \Lambda_H (B)$ of the FBM with respect to Brownian motion $B$, where $\Lambda_H$ is a suitable bounded linear operator from the space of $\lambda$-H\"older continuous functions (with $\frac{1}{2}-H < \lambda < \frac{1}{2}$) to the space of continuous functions equipped with the sup norm. The representation (see Theorem \ref{repTH}) is a simple consequence of the classical Volterra-type representation. Our approximation is a functional of $\Lambda_H$ applied to a skeletal continuous-time random walk $A^k$ (previously suggested by F. Knight \cite{knight1}) imbedded on a given Brownian motion $B$ which satisfies

\begin{equation}\label{epsilonp}
\sup_{0\le t \le T}|A^k(t) - B(t)|\le \epsilon_k~a.s
\end{equation}
for a given sequence $\{\epsilon_k; k\ge 1\}$ such that $\epsilon_k \downarrow 0$ as $k\rightarrow +\infty$.

For a given sequence $\{\epsilon_k; k\ge 1\}$ realizing (\ref{epsilonp}), our approximation scheme admits a rate of convergence of order $O(\epsilon_k^{p(1-2\lambda)+ 2(\delta-1)})$ whenever $\max\{0,1-\frac{pH}{2}\}< \delta < 1$, $\lambda \in \big(\frac{1-H}{2}, \frac{1}{2} + \frac{\delta-1}{p}\big)$. From the perspective of numerical analysis, one advantage of our approximation scheme is the possibility to simulate FBM by only simulating the first time Brownian motion hits $\pm 1$ (see e.g \cite{Burq_Jones2008,milstein}) and a Bernoulli random variable which must be composed with suitable singular deterministic integrals. From a theoretical perspective, such type of approximation plays a key role in the stochastic analysis of processes adapted to FBM via the methodology presented in Le\~ao, Ohashi and Simas \cite{LEAO_OHASHI2017.1} in the context of functional stochastic calculus. In particular, the main result of this short note (Theorem \ref{FBMappHless0.5}) is one of the key arguments to tackle non-Markovian optimal stochastic control problems driven by FBM in the rough regime $0< H < \frac{1}{2}$ as showed in Theorems 6.2 and 6.3 in Le\~ao, Ohashi and Souza \cite{LOF2020}.

We stress that a random walk-type (almost sure) approximation based on Mandelbrot-van Ness representation was studied by Szabados \cite{szabados}. He gave an approximation of FBM for $H \in (\frac{1}{4},1)$ with convergence rate $O(n^{-\min(H-\frac{1}{4}, \frac{1}{4})~2\text{log}2} \text{log}~n)$ at the nth approximation step. Moreover, his convergence is established with respect to some FBM. In contrast, we study the problem with respect to a given FBM with $H \in (0,\frac{1}{2})$ and we are interested in $L^p$ estimates in the supremum norm. Finally, we remark that the scheme introduced in this article is expected to work for approximations of the FBM in the $p$-variation topology for $p> \frac{1}{H} > 2$. We leave this investigation to a future project.

The remainder of this note is organized as follows. Section \ref{pathwiseFBMsection} presents a pathwise representation of FBM with $H \in (0,\frac{1}{2})$ which is an important step in our approximation. Section \ref{mainsection} presents the proof of the main result of this note, namely Theorem \ref{FBMappHless0.5}.

\section{A pathwise representation of FBM with $H \in (0,\frac{1}{2})$}\label{pathwiseFBMsection}
Throughout this note, $(\Omega, \mathbb{F},\mathbb{P})$ denotes a filtered probability space equipped with a one-dimensional standard Brownian motion $B$ where $\mathbb{F}:=(\mathcal{F}_t)_{t\ge 0}$ is the usual $\mathbb{P}$-augmentation of the filtration generated by $B$ under a fixed probability measure $\mathbb{P}$. For a real-valued function $f:[0,T]\rightarrow \mathbb{R}$, we denote

$$\|f\|_\infty:=\sup_{0\le t\le T}|f(t)|,$$
where $ 0 < T < +\infty$ is a fixed terminal time.

In the sequel, we derive a pathwise FBM representation for $0 < H < \frac{1}{2}$ which will play a key role in constructing our approximation scheme for FBM. It is a well-known fact that the FBM can be represented w.r.t a Brownian motion $B$ as follows

$$
\int_0^t K_H(t,s)dB(s); 0\le t\le T,
$$
where $K_H$ is a deterministic kernel described by

%\begin{equation}\label{originalKH}
$$K_H(t,s):=c_H\Bigg[ t^{H-\frac{1}{2}}s^{\frac{1}{2}-H}(t-s)^{H-\frac{1}{2}} - \left(H-\frac{1}{2}\right)s^{\frac{1}{2}-H}\int_s^t u^{H-\frac{3}{2}}(u-s)^{H-\frac{1}{2}}du\Bigg]; 0< s < t,$$
%\end{equation}
for a positive constant $c_H$. Let us define

$$
K_{H,1}(t,s):=c_H t^{H-\frac{1}{2}}s^{\frac{1}{2}-H}(t-s)^{H-\frac{1}{2}},
$$
$$K_{H,2}(t,s):=c_H\left(\frac{1}{2}-H\right)s^{\frac{1}{2}-H}\int_s^t u^{H-\frac{3}{2}}(u-s)^{H-\frac{1}{2}}du,
$$
for $0< s < t$. By definition, $K_H(t,s) = K_{H,1}(t,s) + K_{H,2}(t,s); 0 < s < t$. By making change of variables $v=\frac{u}{s}$, we can write

\begin{equation}\label{solveint}
\int_s^t u^{H-\frac{3}{2}}(u-s)^{H-\frac{1}{2}}du= \theta_H(t,s) s^{2H-1},
\end{equation}
where $\theta_H(t,s):=\int_1^{\frac{t}{s}}v^{H-\frac{3}{2}}(v-1)^{H-\frac{1}{2}}dv$, for $0< s\leq t$. We observe

\begin{equation}\label{solveint1}
\int_1^{+\infty}v^{H-\frac{3}{2}}(v-1)^{H-\frac{1}{2}}dv < \infty,
\end{equation}
for $0 < H < \frac{1}{2}$. Therefore,
$$K_{H,2}(t,s)=c_H\left(\frac{1}{2}-H\right)s^{H-\frac{1}{2}}\int_1^{\frac{t}{s}}v^{H-\frac{3}{2}}(v-1)^{H-\frac{1}{2}}dv$$
and

\begin{equation}\label{partialderK}
\begin{split}
\partial_s K_{H,1}(t,s) &= c_H\left(\frac{1}{2}-H\right)\Big[t^{H-\frac{1}{2}}s^{-\frac{1}{2}-H}(t-s)^{H-\frac{1}{2}} +t^{H-\frac{1}{2}}s^{\frac{1}{2}-H}(t-s)^{H-\frac{3}{2}} \Big],\\
\partial_s K_{H,2}(t,s) &=c_H\left(\frac{1}{2}-H\right) \Bigg[ -t^{H-\frac{1}{2}}s^{-\frac{1}{2}-H}(t-s)^{H-\frac{1}{2}} +\left(H-\frac{1}{2}\right)s^{H-\frac{3}{2}}\theta_H(t,s)\Bigg],
\end{split}
\end{equation}
for $0 < s < t$. In the sequel, if $f:[0,T]\rightarrow \mathbb{R}$, we denote

$$\|f\|_{\lambda}:=\sup_{0 < s < r\le T }\frac{|f(r) - f(s)|}{|r-s|^\lambda},$$
for $0 < \lambda\le 1$. Let $\mathbf{C}^\lambda_0$ be the space of all H\"{o}lder continuous functions $f:[0,T]\rightarrow\mathbb{R}$ with $f(0)=0$ equipped with the norm $\|\cdot\|_\lambda$. For each $f\in \mathbf{C}^\lambda_0$, we define

$$(\Lambda_Hf)(t):=\int_0^t [f(t)-f(s)]\partial_s K_{H,1}(t,s)ds - \int_0^t \partial_s K_{H,2}(t,s)f(s)ds; 0\le t\le T.$$
\begin{lemma}\label{boundHless}
If $\frac{1}{2}-H<\lambda<\frac{1}{2}$, there exists a constant $C$ which depends on $H$ such that

$$\sup_{0\leq t\leq T}|(\Lambda_Hf)(t)|\leq C T^{H-\frac{1}{2}+\lambda}\|f\|_{\lambda},$$
for every $f\in \mathbf{C}^\lambda_0$.
\end{lemma}
\begin{proof}
Just observe there exists a constant $C$ (which only depends on $H$) such that the following estimates hold:
$$
\sup_{0\le t\le T}\int_0^t|\partial_s K_{H,1}(t,s)||f(t)-f(s)|ds \leq C\|f\|_\lambda T^{H+\lambda-\frac{1}{2}}
$$
and
$$\sup_{0\le t\le T}\int_0^t|\partial_s K_{H,2}(t,s)(t,s)||f(s)|ds\le  C \|f\|_{\lambda} T^{H+\lambda-\frac{1}{2}},
$$
whenever $\frac{1}{2}-H<\lambda<\frac{1}{2}$.
\end{proof}

We are now able to prove a pathwise representation for the FBM with $0 < H < \frac{1}{2}$ with respect to a given standard Brownian motion.
\begin{theorem}\label{repTH}
Any FBM with exponent $0 < H < \frac{1}{2}$ on a time interval $[0,T]$ can be represented by $(\Lambda_H B)$ for a real-valued standard Brownian motion $B$.
\end{theorem}
\begin{proof}
We fix $0 < a < b < t$. We recall that any FBM $B_H$ with exponent $0 < H < \frac{1}{2}$ can be represented by $B_H(t) = \int_0^t K_H(t,s)dB(s)$ for some Brownian motion $B$. It is sufficient to check the following identity

\begin{equation}\label{ref1}
\mathbb{E}\Bigg[ (B(b) - B(a))\int_0^t K_H(t,s)dB(s)\Bigg] = \mathbb{E}\Big[ (B(b) - B(a))(\Lambda_H B)(t)\Big].
\end{equation}
In one hand, by It\^o's isometry, the left hand side of (\ref{ref1}) equals to $\int_a^b K_H(t,s)ds = \int_a^b K_{H,1}(t,s)ds + \int_a^b K_{H,2}(t,s)ds$ and integration by parts yields

\begin{eqnarray}
\nonumber\int_a^b K_H(t,s)ds &=& (b-a)K_{H,1}(t,a) + \int_a^b (b-s)\partial_sK_{H,1}(t,s)ds\\
\label{ref0}&+&(b-a)K_{H,2}(t,b) - \int_a^b (s-a)\partial_s K_{H,2}(t,s)ds.
\end{eqnarray}
On the other hand, by using integration by parts (in the sense of Malliavin calculus), we observe the right-hand side of (\ref{ref1}) equals to

\begin{eqnarray}
\nonumber\mathbb{E}\int_a^b \mathbf{D}_r (\Lambda_H B)(t)dr&=& \int_0^t \int_a^b \mathds{1}_{(s,t]}(r)\partial_s K_{H,1}(t,s)dr ds\\
\label{ref2}&-& \int_0^t \int_a^b \partial_s K_{H,2}(t,s)\mathds{1}_{[0,s]}(r)dr ds,
\end{eqnarray}
where $\mathbf{D}$ denotes the Gross-Sobolev-Malliavin derivative. Lastly, we observe the right-hand side of (\ref{ref2}) equals to (\ref{ref0}). This concludes the proof.
\end{proof}

\section{An $L^p$ uniform approximation for FBM in terms of a continuous-time random walk}\label{mainsection}

In this section, we present the rate of convergence of our approximation scheme w.r.t a given FBM $B_H$ with exponent $0 < H < \frac{1}{2}$. For this purpose, we make use of Theorem \ref{repTH} as follows. For a Brownian motion $B$ realizing $B_H = \Lambda_H B$ via Theorem \ref{repTH}, we construct a class of pure jump processes driven by suitable waiting times which describe its local behavior: we set $T^{k}_0:=0$ and

\begin{equation}\label{stopping_times}
T^{k}_n := \inf\{T^{k}_{n-1}< t <\infty;  |B(t) - B(T^{k}_{n-1})| = \varepsilon_k\}, \quad n \ge 1,
\end{equation}
where $\{\epsilon_k; k\ge 1\}$ is an arbitrary sequence such that $\epsilon_k \downarrow 0$ as $k\rightarrow +\infty$. The strong Markov property yields the family $(T^{k}_n)_{n\ge 0}$ is a sequence of stopping times where the increments $\{\Delta T^{k}_n:= T^k_n - T^k_{n-1} ; n\ge 1\}$ is an i.i.d sequence with the same distribution as $T^{k}_1$. By the Brownian scaling property, $\Delta T^{k}_1 = \epsilon^2_k \tau$ (in law) where $\tau = \inf\{t>0; |Y(t)|=1\}$ ($Y$ is a standard Brownian motion) is an absolutely continuous variable with mean equals one and with all finite moments (see Section 5.3.2 in \cite{milstein}). Then, we define the continuous-time random walk $A^{k}$

$$
A^{k} (t) := \sum_{n=1}^{\infty} \big(B(T^k_n) - B(T^k_{n-1}) \big) 1\!\!1_{\{T^{k}_n\leq t \}};~t\ge0.
$$
%One can easily check that $\{\sigma^{k}_n; n\ge 1\}$ is an i.i.d sequence of $\frac{1}{2}$-Bernoulli random variables for each $k\ge 1$. Moreover, $\{\Delta A^{k}(T^{k}_n); n\ge 1\}$ is independent from $\{\Delta T^{k}_n; n\ge 1\}$ for each $k\ge 1$.
By construction

\begin{equation}\label{akb}
\sup_{t\ge 0}|A^{k}(t) - B(t)|\le \epsilon_k~a.s
\end{equation}
for every $k\ge 1$. In the sequel, we set

$$\bar{t}_k := \max\{T^k_n; T^k_n \le t\},$$
and we define
\begin{equation}\label{tmaismenos}
\bar{t}^{+}_k:= \min\{T^{k}_n ; \bar{t}_k < T^{k}_n\}\wedge T~\text{and}~\bar{t}^{-}_k:=\max\{T^{k}_n; T^{k}_n<\bar{t}_k\}\vee 0,
\end{equation}
where we set $\max \emptyset=-\infty$. By construction, $\bar{t}_k \le t < \bar{t}^{+}_k$ a.s for each $t\ge 0$. Let us define

$$B^{k}_H(t):=  \int_0^{\bar{t}_k} \partial_s K_{H,1}(\bar{t}_k,s)\big[A^{k}(\bar{t}_k) - A^{k}(\bar{s}^{+}_k)\big]ds - \int_0^{\bar{t}_k}\partial_s K_{H,2}(\bar{t}_k,s)A^{k}(s)ds; 0\le t\le T.$$
Clearly, $B_H^k$ is a pure jump process of the form

$$B^k_H(t) = \sum_{n=0}^\infty B^k_H(T^k_n)\mathds{1}_{\{T^k_n \le t < T^k_{n+1}\}}; 0\le t\le T.$$

Let us define
$$\|A^{k} - B\|_{-,\lambda}:=\sup_{0\leq s\leq \bar{t}^{-}_k<\bar{t}_k\leq t\leq T}\frac{|A^{k}(\bar{s}^{+}_k)-B(s)|}{(\bar{t}_k-s)^\lambda},
$$

$$\|A^{k} - B\|_{T^{k}_1,\lambda}:=\sup_{T^{k}_1\wedge T<s\leq T}\frac{|A^{k}(s)-B(s)|}{s^\lambda}.$$

\begin{lemma}\label{pathwiseestimate}
If $\frac{1}{2}-H < \lambda < \frac{1}{2}$ and $0 < \varepsilon < H$, then there exists a constant $C$ which only depends on $H$ such that

\begin{eqnarray*}
\|B^{k}_H - B_H\|_\infty&\le& C \|A^{k} - B\|_{-,\lambda}  T^{H-\frac{1}{2}+\lambda} + C \|B\|_\lambda (\max_{m\ge 1}\Delta T^{k}_n)^{H-\frac{1}{2}+\lambda}\mathds{1}_{\{T^{k}_n \le T\}}\\
& &\\
&+&C \left(\|B\|_\lambda\left(T^{k}_1\wedge T\right)^{\lambda+H-\frac{1}{2}} +  \|A^{k} - B\|_{T^{k}_1,\lambda}  T^{\lambda+H-\frac{1}{2}}\right)\\
& &\\
&+& \|B_H\|_{H-\varepsilon} (\max_{n\ge 1} \Delta T^{k}_n\mathds{1}_{\{T^{k}_n \le T\}})^{H-\varepsilon}~a.s,
\end{eqnarray*}
for every $k\ge 1$.
\end{lemma}
\begin{proof}
In the sequel, $C$ is a constant which may differ from line to line and we fix $\frac{1}{2}-H < \lambda < \frac{1}{2}, 0 < \varepsilon < H$. First of all, we observe

$$
|B^{k}_H(t) - B_H(t)|\le |B^{k}_H(t) - B_H(\bar{t}_k)| + \|B_H\|_{H-\varepsilon} (\max_{n\ge 1} \Delta T^{k}_n\mathds{1}_{\{T^{k}_n \le T\}})^{H-\varepsilon}~a.s.
$$
Then,

\begin{equation}\label{in}
\|B^{k}_H - B_H\|_\infty \le\sup_{0\le t\le T}|B^{k}_H(t) - B_H(\bar{t}_k)| + \|B_H\|_{H-\varepsilon} (\max_{n\ge 1} \Delta T^{k}_n\mathds{1}_{\{T^{k}_n\le T\}})^{H-\varepsilon}~a.s.
\end{equation}
To keep notation simple, we denote

\begin{equation*}
\begin{split}
&\varphi^k(t,s):=A^{k}(t)- A^{k}(\bar{s}^{+}_k) - (B(t) - B(s)),\quad \varphi^k(s):=A^{k}(s) - B(s)\\
&\|A^{k} - B\|_{+,\lambda}:=\sup_{\bar{t}^{-}_k< s< \bar{t}_k\leq t\leq T}\frac{|B(\bar{t}_k)-A^{k}(t)-B(s)+A^{k}(\bar{s}^{+}_k)|}{(\bar{t}_k-s)^\lambda} = \sup_{\bar{t}^{-}_k< s< \bar{t}_k\leq t\leq T}\frac{|A^{k}(\bar{s}^{+}_k)-B(s)|}{(\bar{t}_k-s)^\lambda} \\
&\|A^{k} - B\|_{T^{k}_1-,\lambda}:=\sup_{0\leq s\leq T^{k}_1\wedge T}\frac{|A^{k}(s)-B(s)|}{s^\lambda}.
\end{split}
\end{equation*}
At first, we observe $\|A^{k} - B\|_{+,\lambda}\le \|B\|_\lambda$ a.s and $\|A^{k} - B\|_{T^{k}_1-,\lambda}\le \|B\|_\lambda$ a.s for every $k\ge 1$. Furthermore, we have

\begin{eqnarray}\label{in0}
\sup_{0\le t\le T}|B^{k}_H(t) - B_H(\bar{t}_k)| &\le& \sup_{T^{k}_1\le t\le T}\int_0^{\bar{t}_k}\big| \partial_s K_{H,1}(\bar{t}_k,s)\big| \big|\varphi^k(\bar{t}_k,s)\big|ds\\
\nonumber&+& \sup_{T^{k}_1\le t\le T}\int_0^{\bar{t}_k} \big|\partial_s K_{H,2}(\bar{t}_k,s)\big|\big|\varphi^k(s)\big|ds~a.s.
\end{eqnarray}

We observe
\begin{eqnarray*}
\int_0^{\bar{t}_k}|\partial_s K_{H,1}(\bar{t}_k,s) |\varphi^k(\bar{t}_k,s)|ds&\le& C \|A^{k} - B\|_{-,\lambda}  (\bar{t}_k)^{H-\frac{1}{2}} \int_0^{\bar{t}^{-}_k}s^{-\frac{1}{2}-H}(\bar{t}_k-s)^{H-\frac{1}{2}+\lambda}ds\\
 & &\\
 &+& C \|A^{k} - B\|_{+,\lambda} (\bar{t}_k)^{H-\frac{1}{2}} \int_{\bar{t}^{-}_k}^{\bar{t}_k}s^{-\frac{1}{2}-H}(\bar{t}_k-s)^{H-\frac{1}{2}+\lambda}ds\\
 & &\\
 &+& C \|A^{k} - B\|_{-,\lambda} (\bar{t}_k)^{H-\frac{1}{2}} \int_0^{\bar{t}^{-}_k}s^{\frac{1}{2}-H}(\bar{t}_k-s)^{H-\frac{3}{2}+\lambda}ds\\
 & &\\
 &+&C \|A^{k} - B\|_{+,\lambda} (\bar{t}_k)^{H-\frac{1}{2}} \int_{\bar{t}^{-}_k}^{\bar{t}_k}s^{\frac{1}{2}-H}(\bar{t}_k-s)^{H-\frac{3}{2}+\lambda}ds\\
 & &\\
 &=:& I^k_1(\bar{t}_k,s) + I^k_2(\bar{t}_k,s) + I^k_3(\bar{t}_k,s) + I^k_4(\bar{t}_k,s)~a.s.
\end{eqnarray*}
The following estimates hold true a.s

\begin{eqnarray}\label{in1}
\nonumber I^k_1(\bar{t}_k,s)&\le& C\|A^{k} - B\|_{-,\lambda} (\bar{t}_k)^{H-\frac{1}{2}} \int_0^{\bar{t}_k}s^{-\frac{1}{2}-H}(\bar{t}_k-s)^{H-\frac{1}{2}+\lambda}ds\\
&\le& C \|A^{k} - B\|_{-,\lambda} T^{H-\frac{1}{2}+\lambda},
\end{eqnarray}

\begin{equation}\label{in2}
I^k_3(\bar{t}_k,s)\le C\|A^{k} - B\|_{-,\lambda} (\bar{t}_k)^{H-\frac{1}{2}} \int_0^{\bar{t}_k}s^{\frac{1}{2}-H}(\bar{t}_k-s)^{H-\frac{3}{2}+\lambda}ds \le C\|A^{k} - B\|_{-,\lambda} T^{H-\frac{1}{2}+\lambda},
\end{equation}
and
\begin{eqnarray}
\nonumber I^k_2(\bar{t}_k,s) + I^k_4(\bar{t}_k,s)&\le& C \|A^{k} - B\|_{+,\lambda} (\bar{t}_k)^{H-\frac{1}{2}}\Bigg( \nonumber\int_{\bar{t}^{-}_k}^{\bar{t}_k}s^{\frac{1}{2}-H}(\bar{t}_k-s)^{H-\frac{3}{2}+\lambda}ds\\
\nonumber & &\\
\nonumber &+&\int_{\bar{t}^{-}_k}^{\bar{t}_k} (s-\bar{t}^{-}_k)^{-\frac{1}{2}-H}(\bar{t}_k-s)^{H-\frac{1}{2}+\lambda}ds\Bigg)\\
\nonumber& &\\
\nonumber&=& C \|A^{k} - B\|_{+,\lambda} (\bar{t}_k)^{H-\frac{1}{2}} \Big( (\bar{t}_k)^{\frac{1}{2}-H} (\Delta \bar{t}_k)^{H-\frac{1}{2}+\lambda} + (\Delta\bar{t}_k)^{\frac{1}{2}-H} (\Delta \bar{t}_k)^{H-\frac{1}{2}+\lambda} \Big)\\
\label{in3}&\le& C \|B\|_\lambda (\max_{m\ge 1}\Delta T^{k}_n)^{H-\frac{1}{2}+\lambda}\mathds{1}_{\{T^{k}_n \le T\}},
\end{eqnarray}
where $\Delta \bar{t}_k:= \bar{t}_k - \bar{t}^{-}_k$. Summing up (\ref{in1}), (\ref{in2}) and (\ref{in3}), we arrive at the following estimate

\begin{equation}\label{in4}
\sup_{T^{k}_1\le t\le T}\int_0^{\bar{t}_k}|\partial_s K_{H,1}(\bar{t}_k,s) |\varphi^k(\bar{t}_k,s)|ds\le C\|A^{k} - B\|_{-,\lambda} T^{H-\frac{1}{2}+\lambda} + C \|B\|_\lambda (\max_{m\ge 1}\Delta T^{k}_n)^{H-\frac{1}{2}+\lambda}\mathds{1}_{\{T^{k}_n \le T\}},
\end{equation}
almost surely for every $k\ge 1$. Let us now estimate the second term in the right-hand side of (\ref{in0}). At first, we notice

\begin{equation}\label{in5}
\begin{split}
\sup_{T^{k}_1\leq t\leq T}\int_0^{\bar{t}_k} \big|\partial_s K_{H,2}(\bar{t}_k,s)\big| |\varphi^k(s)|ds &\leq C \sup_{T^{k}_1\leq t\leq T}(\bar{t}_k)^{H-\frac{1}{2}}\int_0^{\bar{t}_k}s^{-\frac{1}{2}-H}(\bar{t}_k-s)^{H-\frac{1}{2}}|\varphi^k(s)|ds\\
&+ C \sup_{T^{k}_1\leq t\leq T}\int_0^{\bar{t}_k}s^{-\frac{1}{2}-H}\int_s^t u^{H-\frac{3}{2}}(u-s)^{H-\frac{1}{2}}du|\varphi^k(s)|ds\\
&=:J_{k,1} + J_{k,2}~a.s,
\end{split}
\end{equation}
where

\begin{equation}\label{in6}
\begin{split}
J_{k,1}
&\leq \sup_{T^{k}_1\leq t\leq T}\Big[(\bar{t}_k)^{H-\frac{1}{2}}\|A^{k} - B\|_{T^{k}_1-,\lambda}\int_0^{T^{k}_1\wedge T}s^{\lambda-\frac{1}{2}-H}(\bar{t}_k-s)^{H-\frac{1}{2}}ds\Big]\\
&+\sup_{T^{k}_1\leq t\leq T} \Big[(\bar{t}_k)^{H-\frac{1}{2}}\|A^{k} - B\|_{T^{k}_1,\lambda} \int_{T^{k}_1\wedge T}^{\bar{t}_k}(s-T^{k}_1)^{\lambda-\frac{1}{2}-H}(\bar{t}_k-s)^{H-\frac{1}{2}}ds\Big]\\
&\leq  \sup_{T^{k}_1\leq t\leq T}(\bar{t}_k)^{H-\frac{1}{2}}\|A^{k} - B\|_{T^{k}_1-,\lambda}\left(T^{k}_1\wedge T\right)^\lambda +\sup_{T^{k}_1\leq t\leq T}\Big[(\bar{t}_k)^{H-\frac{1}{2}}\left(\bar{t}_k -T^{k}_1\wedge T \right)^\lambda\|A^{k} - B\|_{T^{k}_1,\lambda}\Big] \\
&\leq C \Big(\|A^{k} - B\|_{T^{k}_1-,\lambda} \left(T^{k}_1\wedge T\right)^{\lambda+H-\frac{1}{2}}+\|A^{k} - B\|_{T^{k}_1,\lambda}T^{\lambda+H-\frac{1}{2}}\Big)~a.s.
\end{split}
\end{equation}

By using (\ref{solveint}) and (\ref{solveint1}), we have

$$\int_s^t u^{H-\frac{3}{2}}(u-s)^{H-\frac{1}{2}}du \le s^{2H-1} \sup_{0 < r < x\le T}\theta_H(x,r) < \infty$$
for every $(s,t); 0 < s < t\le T $. Then, there exists a constant $C$ which depends on $H$ such that

\begin{eqnarray}
\nonumber J_{k,2} &\leq& C \left(\|A^{k} - B\|_{T^{k}_1-,\lambda}\int_0^{T^{k}_1\wedge T}s^{\lambda+H-\frac{3}{2}}ds + \|A^{k} - B\|_{T^{k}_1,\lambda}\sup_{T^{k}_1\leq t\leq T}\int_{T^{k}_1\wedge T}^{\bar{t}_k\wedge T}s^{\lambda+H-\frac{3}{2}}ds\right)\\
\label{in7}& &\\
\nonumber&\leq& C \left(\|A^{k} - B\|_{T^{k}_1-,\lambda}\left(T^{k}_1\wedge T\right)^{\lambda+H-\frac{1}{2}} + \|A^{k} -  B\|_{T^{k}_1,\lambda}\sup_{T^{k}_1\leq t\leq T}\left(\bar{t}_k -T^{k}_1\wedge T \right)^{\lambda+H-\frac{1}{2}}\right)~a.s.
\end{eqnarray}
By the estimates (\ref{in5}), (\ref{in6}) and (\ref{in7}) and using the fact $\|A^{k} - B\|_{T^{k}_1-,\lambda}\le \|B\|_\lambda$ a.s for every $k\ge 1$, we arrive at the inequality

\begin{equation}\label{i8}
\sup_{T^{k}_1\leq t\leq T}\int_0^{\bar{t}_k} \big|\partial_s K_{H,2}(\bar{t}_k,s)\big| |\varphi^k(s)|ds\le C \Big(\|B\|_{\lambda} \left(T^{k}_1\wedge T\right)^{\lambda+H-\frac{1}{2}}+\|A^{k} - B\|_{T^{k}_1,\lambda} T^{\lambda+H-\frac{1}{2}}\Big)~a.s.
\end{equation}
 Summing up (\ref{in}), (\ref{in0}), (\ref{i8}) and (\ref{in4}), we conclude the proof.
\end{proof}

In order to establish the main result of this paper (namely Theorem \ref{FBMappHless0.5}), we make a fundamental use of Lemma 2.2 in \cite{LEAO_OHASHI2017.1} but in a slightly different way. In \cite{LEAO_OHASHI2017.1}, the authors establish an upper bound

\begin{equation}\label{uppermesh}
\mathbb{E}\Big| \max_{n\ge 1} \Delta T^k_n\Big|^q\mathds{1}_{\{T^k_n \le T\}}\lesssim \epsilon^{2q}_k \lceil \epsilon_k^{-2}T\rceil^{1-\alpha};~k\ge 1,
\end{equation}
where $q \ge 1$, $\lceil \cdot \rceil$ denotes the ceiling function and $\alpha\in (0,1)$ is a constant. Here, we will establish an upper bound for
\begin{equation}\label{uppermesh1}
\mathbb{E}\Big| \max_{n\ge 1} \frac{1}{\Delta T^k_n}\Big|^q\mathds{1}_{\{T^k_n \le T\}}; k\ge 1,
\end{equation}
which obviously blows up as $k\rightarrow +\infty$. In the present context, we need to know precisely how fast (\ref{uppermesh1}) blows up. Since $\max_{n\ge 1}(\Delta T^k_n)^{-1}\mathds{1}_{\{T^k_n\le T\}}$ is an unbounded sequence, the proof of Lemma 2.2 in \cite{LEAO_OHASHI2017.1} does not apply directly to the case (\ref{uppermesh1}). In the sequel, we give the details of the obtention of the upper bound for (\ref{uppermesh1}). At first, we recall the following elementary result.

\begin{lemma}\label{aux}
Let $Z_1, \ldots, Z_n$ be a sequence of positive random variables on a probability space. Then, for every $q\ge 1, r>1$, we have

$$\mathbb{E}\big(\max_{m\le i\le n}Z_i\big)^q \le \Bigg\{ \sum_{\ell=m}^n \mathbb{E}|Z_\ell|^{qr} \Bigg\}^{\frac{1}{r}}.$$

\end{lemma}
\begin{proof}
Just apply H\"older's inequality as follows

\begin{eqnarray*}
\mathbb{E}\big(\max_{m\le i\le n}Z_i\big)^q  &=& \mathbb{E}\max_{m\le i\le n}|Z_i|^q  = \mathbb{E}\Big\{\max_{m\le i\le n}|Z_i|^{qr}\Big\}^{\frac{1}{r}}\\
&\le& \Big\{\mathbb{E}\max_{m\le i\le n}|Z_i|^{qr}\Big\}^{\frac{1}{r}}\\
&\le& \Big\{\mathbb{E}\sum_{i=m}^n|Z_i|^{qr}\Big\}^{\frac{1}{r}}.
\end{eqnarray*}
\end{proof}

\begin{lemma}\label{meshlemma}
For every $q\ge 1$ and $\alpha \in (0,1)$, there exists a constant $C$ which depends on $q$ and $\alpha$ such that

$$\mathbb{E}\Big| \max_{n\ge 1} \frac{1}{\Delta T^k_n}\Big|^q\mathds{1}_{\{T^k_n \le T\}} \le C \epsilon_k^{-2q}\lceil \epsilon_k^{-2}T\rceil^{1-\alpha},$$
for every $k\ge 1$.
\end{lemma}
\begin{proof}
Let $\tau = \inf\{t>0; |B_t|=1\}$. We recall (see e.g Lemmas 2 and 3 in \cite{Burq_Jones2008}) that, for every $q>0$, we have $
\mathbb{E} \tau^{-q} < \infty$ and $\mathbb{E}[\tau] = 1$. Let $N^k_T$ be the number of jumps of $A^k$ along the time interval $[0,T]$. In other words,

$$N^k_T = \sum_{n\ge 1}\mathds{1}_{\{T^k_n \le T\}}.$$
We observe $N^k_T = \epsilon^{-2}_k [A^k, A^k]_T$, where $[\cdot, \cdot]_T$ denotes the quadratic variation of the martingale $A^k$ computed w.r.t its own filtration. Moreover, $\{N^k_T = n\} = \{T^k_n < T < T^k_{n+1}\}~a.s$ for every $k,n\ge 1$. Recall
 $\{\Delta T^k_i; i\ge 1\}$ is an iid sequence with absolutely continuous distribution and $\Delta T^k_1 =\epsilon_k^2 \tau$ (in law). Since $T^k_n = \sum_{j=1}^n \Delta T^k_j, T^k_{n+1} = T^k_n + \Delta T^k_{n+1}$ and $\Delta T^k_{n+1}$ is independent from $T^k_n$, then $\{N^k_T = n\}$ has strictly positive probability for every $k,n\ge 1$. Let $\mathbb{E}_{k,n}$ be the expectation computed w.r.t the probability measure $\mathbb{P}[\cdot | N^k_T = n]$ for integers $n,k\ge 1$. We observe $\Delta T^k_i; 1\le i\le n$ is an identically distributed sequence conditioned on the event $\{N^k_T = n\}$, i.e.,

\begin{eqnarray}\label{appl1}
\nonumber\mathbb{P} \big\{ \Delta T^k_i \in dx | N^k_T = n\big\} &=& \mathbb{P} \big\{ \Delta T^k_1 \in dx | N^k_T = n\big\}\\
&=& \mathbb{P}\{\epsilon^{2}_k \tau \in dx| N^k_T = n\}
\end{eqnarray}
for each $i\in \{1, \ldots, n\}$.

By Lemma \ref{aux}, we know that for a given $\alpha \in (0,1)$ and $1\le m < n$

\begin{equation}\label{without1}
\mathbb{E}\max_{m\le i\le n} \Big(\frac{1}{\Delta T^k_i}\Big)^{q} \le \Bigg\{\sum_{i=m}^n \mathbb{E}\Big( \frac{1}{\Delta T^k_i}\Big)^{\frac{q}{1-\alpha}} \Bigg\}^{1-\alpha} = \epsilon^{-2q}_k (n-m+1)^{1-\alpha}(\mathbb{E}[\tau^{\frac{-q}{1-\alpha}}])^{1-\alpha}.
\end{equation}
In addition, property (\ref{appl1}) and Lemma \ref{aux} yield

\begin{eqnarray}\label{ref3}
\mathbb{E}\Big[ \Big|\max_{m\le i \le n} \frac{1}{\Delta T^k_i}\Big|^q \Big| N^k_T=n\Big] &=& \mathbb{E}_{k,n}\Big|\max_{m\le i \le n} \frac{1}{\Delta T^k_i}\Big|^q\le \Bigg\{\sum_{i=m}^n \mathbb{E}_{k,n}\Big( \frac{1}{\Delta T^k_i}\Big)^{\frac{q}{1-\alpha}} \Bigg\}^{1-\alpha} \\
\nonumber&\le& \Bigg( \mathbb{E}_{k,n}\big[ (T_1^k)^{\frac{-q}{1-\alpha}}\big]\Bigg)^{1-\alpha} (n-m+1)^{1-\alpha}\\
\nonumber&=& \epsilon^{-2q}_k\big(\mathbb{E}_{k,n}[\tau^{\frac{-q}{1-\alpha}}]\big)^{1-\alpha}(n-m+1)^{1-\alpha}.
\end{eqnarray}
By (\ref{without1}), we get

\begin{equation}\label{ref5}
\mathbb{E} \Big|\max_{1\le m \le \lceil \epsilon^{-2}_k T\rceil } \frac{1}{\Delta T^k_m}\Big|^q \le\epsilon^{-2q}_k\big(\mathbb{E}[\tau^{\frac{-q}{1-\alpha}}]\big)^{1-\alpha}\lceil \epsilon^{-2}_k T\rceil^{1-\alpha},
\end{equation}
for every $k\ge 1$. Then, (\ref{ref5}) yields

\begin{eqnarray}\label{ref6}
\nonumber\mathbb{E}\Big| \max_{m\ge 1} \frac{1}{\Delta T^k_m}\Big|^q\mathds{1}_{\{T^k_m \le T\}} &\le& \mathbb{E}\max \Bigg\{ \Big|\max_{1\le m \le \lceil \epsilon^{-2}_k T\rceil} \frac{1}{\Delta T^k_m}\Big|^q, \Big|\max_{\lceil \epsilon^{-2}_kT \rceil \le m \le \lceil \epsilon^{-2}_k T\rceil \vee N^k_T } \frac{1}{\Delta T^k_m}\Big|^q\Bigg\}\\
\nonumber&\le& \mathbb{E}\Big|\max_{1\le m \le \lceil \epsilon^{-2}_k T\rceil} \frac{1}{\Delta T^k_m}\Big|^q + \mathbb{E}\Big|\max_{\lceil \epsilon^{-2}_kT \rceil\le m \le \lceil \epsilon^{-2}_k T\rceil \vee N^k_T } \frac{1}{\Delta T^k_m}\Big|^q\\
\nonumber&\le& \epsilon^{-2q}_k\big(\mathbb{E}[\tau^{\frac{-q}{1-\alpha}}]\big)^{1-\alpha}\lceil \epsilon^{-2}_k T\rceil^{1-\alpha}\\
&+& \mathbb{E}\Big|\max_{\lceil \epsilon^{-2}_kT \rceil\le m \le \lceil \epsilon^{-2}_k T\rceil \vee N^k_T } \frac{1}{\Delta T^k_m}\Big|^q,
\end{eqnarray}
for every $k\ge 1$. By using (\ref{ref5}) again, we observe

\begin{eqnarray}
\nonumber\mathbb{E}\Big|\max_{\lceil \epsilon^{-2}_kT \rceil\le m \le \lceil \epsilon^{-2}_k T\rceil \vee N^k_T } \frac{1}{\Delta T^k_m}\Big|^q&=& \int_{\{N^k_T \le  \lceil \epsilon^{-2}_kT \rceil\} }\Big|\max_{\lceil \epsilon^{-2}_kT \rceil\le m \le \lceil \epsilon^{-2}_k T\rceil \vee N^k_T } \frac{1}{\Delta T^k_m}\Big|^q d\mathbb{P}\\
\nonumber&+& \int_{\{\lceil \epsilon^{-2}_kT \rceil < N^k_T \le 2\lceil \epsilon^{-2}_kT \rceil\} }\Big|\max_{\lceil \epsilon^{-2}_kT \rceil\le m \le \lceil \epsilon^{-2}_k T\rceil \vee N^k_T } \frac{1}{\Delta T^k_m}\Big|^q d\mathbb{P}\\
\nonumber&+& \int_{\{N^k_T > 2\lceil \epsilon^{-2}_kT \rceil\} }\Big|\max_{\lceil \epsilon^{-2}_kT \rceil\le m \le \lceil \epsilon^{-2}_k T\rceil \vee N^k_T } \frac{1}{\Delta T^k_m}\Big|^q d\mathbb{P}\\
\nonumber&\le& \epsilon^{-2q}_k\big(\mathbb{E}[\tau^{\frac{-q}{1-\alpha}}]\big)^{1-\alpha}\lceil \epsilon^{-2}_k T\rceil^{1-\alpha} + \mathbb{E}\Big|\max_{\lceil \epsilon^{-2}_kT \rceil\le m \le 2 \lceil \epsilon^{-2}_k T\rceil } \frac{1}{\Delta T^k_m}\Big|^q\\
\nonumber&+&\int_{\{N^k_T > 2\lceil \epsilon^{-2}_kT \rceil\} }\Big|\max_{\lceil \epsilon^{-2}_kT \rceil\le m \le \lceil \epsilon^{-2}_k T\rceil \vee N^k_T } \frac{1}{\Delta T^k_m}\Big|^qd\mathbb{P}\\
\nonumber&\le& 2 \epsilon^{-2q}_k\big(\mathbb{E}[\tau^{\frac{-q}{1-\alpha}}]\big)^{1-\alpha}\lceil \epsilon^{-2}_k T\rceil^{1-\alpha}\\
\label{ref7}&+& \int_{\{N^k_T > 2\lceil \epsilon^{-2}_kT \rceil\} }\Big|\max_{\lceil \epsilon^{-2}_kT \rceil\le m \le \lceil \epsilon^{-2}_k T\rceil \vee N^k_T } \frac{1}{\Delta T^k_m}\Big|^qd \mathbb{P},
\end{eqnarray}
for every $k\ge 1$. By (\ref{ref3}), we observe

$$\int_{N^k_T > 2\lceil \epsilon^{-2}_kT \rceil }\Big|\max_{\lceil \epsilon^{-2}_kT \rceil\le m \le \lceil \epsilon^{-2}_k T\rceil \vee N^k_T } \frac{1}{\Delta T^k_m}\Big|^qd \mathbb{P}
  $$
  $$= \int_{2\lceil \epsilon^{-2}_k T\rceil+1} \mathbb{E}\Big[ \Big|\max_{ \lceil \epsilon^{-2}_k T\rceil\le m \le \lceil \epsilon^{-2}_k T\rceil\vee i} \frac{1}{\Delta T^k_m}\Big|^q\ \Big| N^k_T=i\Big]d\mathbb{P}_{N^k_T}(di)$$
$$\le  \epsilon^{-2q}_k\sum_{i\ge 2\lceil \epsilon^{-2}_k T\rceil+1} \Bigg( \mathbb{E}_{k,i}\big[ \tau^{\frac{-q}{1-\alpha}}\big]\Bigg)^{1-\alpha} (i-\lceil \epsilon^{-2}_k T\rceil+1)^{1-\alpha} \mathbb{P}\{N^k_T = i\}$$
$$= \epsilon^{-2q}_k\sum_{i\ge 2\lceil \epsilon^{-2}_k T\rceil+1} \Bigg( \int_0^\infty x^{\frac{-q}{1-\alpha}}\mathbb{P}\{ \tau\in dx, N^k_T = i\}\Bigg)^{1-\alpha}  (i-\lceil \epsilon^{-2}_k T\rceil+1)^{1-\alpha} (\mathbb{P}\{N^k_T = i\})^{\alpha} $$

\begin{equation}\label{ref8}
\le \epsilon^{-2q}_k  (\mathbb{E}[\tau^{\frac{-q}{1-\alpha}}])^{1-\alpha}  \sum_{j\ge 1}   (\lceil \epsilon^{-2}_k T\rceil+j+1)^{1-\alpha} (\mathbb{P}\{N^k_T = 2\lceil \epsilon^{-2}_k T\rceil+j\})^{\alpha},
\end{equation}
for every $k\ge 1$. Now, by applying a standard Large Deviation estimate, we know that

\begin{eqnarray}
\nonumber\mathbb{P}\{N^k_T = 2\lceil \epsilon^{-2}_k T\rceil+j\}&\le& \mathbb{P}\big\{N^k_T \ge  2\lceil \epsilon^{-2}_k T\rceil+j\}\\
\nonumber&=& \mathbb{P}\{T^k_{2\lceil \epsilon^{-2}_k T\rceil+j} \le T\big\}\\
\nonumber &=& \mathbb{P}\Bigg\{\frac{1}{2\lceil \epsilon^{-2}_k T\rceil +j} \sum_{i=1}^{2\lceil \epsilon^{-2}_k T\rceil +j} \tau_i\le \frac{T}{\epsilon^2_k (2\lceil \epsilon^{-2}_k T\rceil +j)}\Bigg\}\\
\nonumber &\le&  \mathbb{P}\Bigg\{\frac{1}{2\lceil \epsilon^{-2}_k T\rceil +j} \sum_{i=1}^{2\lceil \epsilon^{-2}_k T\rceil +j} \tau_i\le \frac{T}{\epsilon^2_k2\lceil \epsilon^{-2}_k T\rceil}\Bigg\}\\
\nonumber &\le& \mathbb{P}\Bigg\{\frac{1}{2\lceil \epsilon^{-2}_k T\rceil +j} \sum_{i=1}^{2\lceil \epsilon^{-2}_k T\rceil +j} \tau_i\le \frac{1}{2}\Bigg\}\\
\label{ref9}&\le& \exp \Big( -( 2\lceil \epsilon^{-2}_k T\rceil+j) w(0.5) \Big),
\end{eqnarray}
where $(\tau_i)_{i=1}^{\infty}$ is an iid sequence such that $\tau_1= \tau$ (in law),  $w(0.5)$ is the Cramer transform of $\tau$ evaluated at the point $0.5$. For each $\alpha \in (0,1)$, there exists $C = C(\alpha)$ such that

\begin{equation}\label{ref10}
\sum_{j=1}^\infty \exp \Big( -\alpha ( 2\lceil \epsilon^{-2}_k T\rceil+j) w(0.5) \Big)j^{1-\alpha}\le C,
\end{equation}
for every $k\ge 1$. For instance, notice that

$$\exp \Big( -\alpha ( 2\lceil \epsilon^{-2}_k T\rceil+j) w(0.5) \Big)\le \frac{1}{j^2} \alpha^{-2},$$
for every $k,j\ge 1$. Summing up (\ref{ref6}), (\ref{ref7}), (\ref{ref8}), (\ref{ref9}) and (\ref{ref10}), we conclude the proof.
\end{proof}

\begin{lemma}\label{touches1}
Fix $p\ge 1$. For each pair $(\lambda, \delta)$ satisfying $0 < \lambda < \frac{1}{2} + \frac{2\delta -2}{2p}, \max\{0,1-\frac{p}{2}\} < \delta < 1$, there exists a constant $C>0$ which depends on $p,\lambda$ and $\delta$, such that

\begin{equation}\label{fi2}
\begin{split}
&\mathbb{E}\big\|A^{k} - B \big\|^p_{-,\lambda}\leq C\epsilon^{p(1-2\lambda)}_k\Big\lceil \epsilon^{-2}_k T\Big\rceil^{1-\delta}\\
\end{split}
\end{equation}
and

\begin{equation}\label{fi4}
\mathbb{E}\big\|A^{k} - B \big\|^p_{T^k_1,\lambda}\le C \epsilon^{p(1-2\lambda)}_k,
\end{equation}
for every $k\ge 1$.
\end{lemma}
\begin{proof}
Fix $p\ge 1$. Let $(\lambda, \delta)$ be a pair of numbers satisfying $0 < \lambda < \frac{1}{2} + \frac{2\delta -2}{2p}, \max\{0,1-\frac{p}{2}\} < \delta < 1$. Notice if $0\leq s\leq \bar{t}^{-}_k$, then $(\bar{t}_k-s)\ge \bar{t}_k-\bar{t}^{-}_k=\Delta \bar{t}_k$. By applying Lemma \ref{meshlemma}, we have

\begin{equation}
\begin{split}
\mathbb{E}\|A^k- B\|^p_{-,\lambda}&\le \mathbb{E}\Bigg(\sup_{0\leq s\leq \bar{t}^{-}_k<\bar{t}_k\leq t\leq T}\frac{|B(s)-A^{k}(\bar{s}^{+}_k)|}{(\Delta \bar{t}_k)^\lambda}\Bigg)^p\\
%&+\mathbb{E}\sup_{0\leq s\leq \bar{t}^-_k<\bar{t}_k\leq t\leq T}\frac{|B(\bar{t}_k)-A^k(t)|}{(\Delta \bar{t}_k)^\lambda}\\
&\leq C \epsilon^p_k\mathbb{E}\Big(\max_{n\ge 1}\left(\frac{1}{\Delta T^{k}_n}\right)\mathds{1}_{\{T^{k}_n\le T\}}\Big)^{\lambda p}\\
&\leq C\left(\epsilon_k^{p-2\lambda p}\Big\lceil \epsilon^{-2}_k T\Big\rceil^{1-\delta}\right),\\
%&=C\epsilon_k^{2\delta-2\lambda-1}
\end{split}
\end{equation}
for a constant $C$ which depends on $\delta, p, \lambda$. This shows (\ref{fi2}). Now, we observe

$$\Bigg(\sup_{T^{k}_1 < s\le T}\frac{|A^{k}(s) - B(s)|}{s^\lambda}\Bigg)^p\le \epsilon^p_k (T^{k}_1)^{-\lambda p}~a.s,$$
for every $k\ge 1$. By definition, $T^{k}_1 = \inf\{t>0; |B(t)| = \epsilon_k\}\stackrel{d}{=}\epsilon^{2}_k \tau$, where $\tau$ is given in the proof of Lemma \ref{meshlemma}. Then,

$$\mathbb{E}(T^{k}_1)^{-\lambda p}\le C \epsilon^{-2\lambda p}_k,$$
for a constant $C$ which depends on $\lambda$ and $p$. We then get

$$\mathbb{E}\|A^k - B\|^p_{T^k_1,\lambda}\le C \epsilon_k^{p(1-2\lambda)},$$
for every $k\ge 1$. This shows (\ref{fi4}) and hence, we conclude the proof.
\end{proof}

\begin{theorem}\label{FBMappHless0.5}
Fix $0 < H < \frac{1}{2}$ and $p\ge 1$. For every pair $(\delta,\lambda)$ such that $\max\{0,1-\frac{pH}{2}\}< \delta < 1$, $\lambda \in \big(\frac{1-H}{2}, \frac{1}{2} + \frac{2\delta-2}{2p}\big)$, there exists a constant $C$ which depends on $p,\delta,H,T,\lambda$ such that

$$
\mathbb{E}\|B^{k}_H - B_H\|^p_\infty \le C \Big(\epsilon_k^{p(1-2\lambda)+ 2(\delta-1)}\Big),
$$
for every $k\ge 1$.
\end{theorem}
\begin{proof}
In the sequel, $C$ is a constant which may differ form line to line. Let us fix $0 < H< \frac{1}{2}, p\ge 1$ and $0 < \varepsilon < H$. By Lemma \ref{pathwiseestimate}, if $\frac{1}{2}-H < \lambda < \frac{1}{2}$, there exists a constant $C$ which depends on $H, T$ and $p\ge 1$ such that

\begin{eqnarray*}
\|B^{k}_H - B_H\|^p_\infty&\le& C \|A^{k} - B\|^p_{-,\lambda} + C \|B\|^p_\lambda (\max_{m\ge 1}\Delta T^{k}_n)^{p(H-\frac{1}{2}+\lambda)}\mathds{1}_{\{T^{k}_n \le T\}}\\
& &\\
&+&C\|B\|^p_\lambda\left(T^{k}_1\wedge T\right)^{p(\lambda+H-\frac{1}{2})} +  C \|A^{k} - B\|^p_{T^{k}_1,\lambda}\\
& &\\
&+& \|B_H\|^p_{H-\varepsilon} (\max_{n\ge 1} \Delta T^{k}_n\mathds{1}_{\{T^{k}_n \le T\}})^{p(H-\varepsilon)}~a.s,
\end{eqnarray*}
for every $k\ge 1$. Now, we notice $\frac{(1-H)}{2} < \frac{1}{2} + \frac{2\delta-2}{2p}$ if, only if, $1-\frac{pH}{2} < \delta < 1$. For each $\lambda, \delta$ satisfying $\frac{(1-H)}{2} < \lambda < \frac{1}{2} + \frac{2\delta -2}{2p}, \max\{0,1-\frac{pH}{2}\} < \delta < 1$, we make use of the Gaussian tails of the Brownian motion and FBM jointly with Lemma \ref{touches1} and (\ref{uppermesh}) to get a constant $C>0$ which depends on $H,T,p, \lambda, \varepsilon$ and $\delta$ such that
$$
\mathbb{E}\|B^{k}_H - B_H\|^p_\infty \le C \epsilon_k^{p(1-2\lambda)}\Big\lceil \epsilon^{-2}_k T\Big\rceil^{1-\delta}  + C \epsilon_k^{2p(H-\frac{1}{2}+\lambda)}\Big\lceil \epsilon^{-2}_k T\Big\rceil^{1-\delta}+ C\epsilon^{2p(H-\varepsilon)}_k\Big\lceil \epsilon^{-2}_k T \Big\rceil^{1-\delta},
$$
for every $k\ge 1$. By noticing that $\lceil x\rceil\le 1 + x$ for every $x\ge 0$, we have
$$
\mathbb{E}\|B^{k}_H - B_H\|^p_\infty \le C \Big(\epsilon_k^{p(1-2\lambda)+ 2(\delta-1)}  + \epsilon_k^{2p(H-\frac{1}{2}+\lambda)+ 2(\delta-1)}+ \epsilon^{2p(H-\varepsilon) + 2(\delta-1)}_k\Big),
$$
for every $k\ge 1$. In fact, whenever the pair $(\delta,\lambda)$ satisfies $\max\{0,1-\frac{pH}{2}\} < \delta < 1$ and $\frac{1-H}{2} < \lambda < \frac{1}{2}+\frac{2(\delta-1)}{2p}$, there exists $C$ such that

$$
\mathbb{E}\|B^{k}_H - B_H\|^p_\infty \le C \Big(\epsilon_k^{p(1-2\lambda)+ 2(\delta-1)}  + \epsilon^{2p(H-\varepsilon) + 2(\delta-1)}_k\Big),
$$
for every $k\ge 1$. Finally, by taking $\varepsilon$ small enough, we conclude the proof.
\end{proof}

\

\textbf{Acknowledgement:} Alberto Ohashi acknowledges the financial support of Math Amsud grant 88887.197425/2018-00 and CNPq-Bolsa de Produtividade de Pesquisa grant 303443/2018-9. Francys A. de Souza acknowledges the financial support of FAPESP 2017/23003-6. The authors would like to thank the Referee for the careful reading and suggestions which considerably improved the presentation of this paper.

%\acks Alberto Ohashi and Francesco Russo acknowledge the financial support of Math Amsud grant 88887.197425/2018-00. Alberto Ohashi acknowledges ENSTA ParisTech for the hospitality and CNPq-Bolsa de Produtividade de Pesquisa grant 303443/2018-9. The authors would like to thank the Referees for the careful reading and suggestions which considerably improved the presentation of this paper.

\end{document}